\documentclass[12pt, letterpaper]{amsart}

\usepackage{amsfonts}
\usepackage{amsmath}
\usepackage{amssymb}
\usepackage{amsthm}
\usepackage{mathscinet}
\usepackage[margin=1in]{geometry}
\usepackage{color,hyperref,url}

\title{Riemann moduli spaces are quantum ergodic}

\author{Dean Baskin}
\address{Department of Mathematics, Texas A\&M University}
\email{dbaskin@math.tamu.edu}
\author{Jesse Gell-Redman}
\address{School of Mathematics and Statistics, University of
  Melbourne}
\email{jgell@unimelb.edu.au}
\author{Xiaolong Han}
\address{Department of Mathematics, California State University,
  Northridge}
\email{xiaolong.han@csun.edu}

\newcommand{\R}{\mathbb R}
\newcommand{\pd}[1][]{\partial_{#1}}
\newcommand{\norm}[2][]{\left\| #2\right\|_{#1}}
\newcommand{\Mreg}{M}
\newcommand{\Mfull}{\overline{M}}
\newcommand{\lap}{\Delta}

\newcommand{\reals}{\mathbb{R}}
\DeclareMathOperator{\Vol}{Vol}
\DeclareMathOperator{\WF}{WF}
\newcommand{\lo}{o}
\DeclareMathOperator{\supp}{supp}
\newcommand{\WP}{\mathrm{WP}}
\newcommand\lra{\longrightarrow}
\newcommand{\reg}{\mathrm{reg}}
\newcommand{\orb}{\mathrm{orb}}

\newtheorem{lemma}{Lemma}[section]
\newtheorem{theorem}[lemma]{Theorem}

\newtheorem{cor}[lemma]{Corollary}
\theoremstyle{remark}
\newtheorem*{rmk}{Remark}

\numberwithin{equation}{section}

\begin{document}

\begin{abstract}
In this note we show that the Riemann moduli spaces
$\mathcal{M}_{\gamma,n}$ equipped with the Weil--Petersson metric are
quantum ergodic for $3\gamma + n \geq 4$.  We also provide other
examples of singular spaces with ergodic geodesic flow for which
quantum ergodicity holds.
\end{abstract}

\maketitle

\section{Introduction}\label{sec:introduction}

The aim of this note is to establish quantum ergodicity on a class of
singular spaces; the main examples we address are the Riemann moduli
spaces $\mathcal{M}_{\gamma,n}$ of Riemann surfaces
of genus $\gamma$ with $n$ marked points equipped with the
Weil-Petersson metric $g_{\WP}$.  We work in the stable range $3\gamma
+ n \geq 4$, so $\mathcal{M}_{\gamma,
  n}$ is a complex orbifold of complex dimension $3 \gamma - 3 + n$
with smooth top dimensional stratum $\mathcal{M}_{\gamma, n, \reg}$. In this setting, we prove the following theorem:
\begin{theorem}[Quantum ergodicity on Riemann moduli spaces]\label{thm:QE-moduli}
Let $3\gamma + n \geq 4$ and $\lap_{g_{\WP}}$ be the positive
Laplacian with respect to the Weil-Petersson metric $g_\WP$ on
$\Mreg=\mathcal{M}_{\gamma,n, \reg}$. Suppose that $\{\phi_{j}\}$ is an
orthonormal basis of eigenfunctions of $\lap_{g_{\WP}}$ on
$\Mreg$ for the natural self-adjoint extension of $\lap_{g_{\WP}}$
  studied by Ji--Mazzeo--M{\"u}ller--Vasy~\cite{JMMV}. Then there is a density one subsequence
$\{\phi_{j_k}\}\subset\{\phi_{j}\}$ such that
\begin{equation*}
\langle A\phi_{j_k}, \phi_{j_k} \rangle \to \int_{S^{*}\Mreg}\sigma_{0}(A)\,d\mu\quad\text{as }k\to\infty
\end{equation*}
for all zero order pseudodifferential operators $A$ with Schwartz
kernel compactly supported in the interior of $\Mreg\times\Mreg$ and
$\sigma_0(A)$ is the principal symbol of $A$. Here, $d\mu$ is the
Liouville measure on the cosphere bundle $S^*\Mreg$ which is
normalized such that $\mu(S^*\Mreg)=1$.
\end{theorem}
In Theorem \ref{thm:stronger} below we prove a stronger result
which allows for pseudodifferential operators $A$
which are supported at the orbifold singularities.

In particular, the above theorem asserts the equidistribution of
``almost all'' eigenfunctions on the Riemann moduli spaces. An
immediate consequence of taking $A=a(x)\in C^\infty_c(\Mreg)$ to
approximate a characteristic function from above and below is that
$$\int_\Omega|\phi_{j_k}|^2\to\frac{\Vol(\Omega)}{\Vol(\Mreg)}\quad\text{as }k\to\infty$$
for all smooth domains $\Omega\Subset\Mreg$. 

The ergodicity of the Weil-Petersson geodesic flow on Riemann moduli
spaces is a celebrated result of Burns--Masur--Wilkinson \cite{BMW}. (See Section \ref{sec:riem-moduli-spac} for more background.) Therefore, the quantum ergodicity in Theorem \ref{thm:QE-moduli} establishes the correspondence of the 
geodesic flow and Laplacian eigenfunctions (which are the stationary states of the quantized operator of the geodesic flow).

Quantum ergodicity was first proved on boundary-less compact manifolds with ergodic geodesic flow by \v Snirel'man \cite{Sn}, Zelditch \cite{Ze}, and Colin de Verdi\`ere \cite{CdV}; on manifolds with boundary, if the billiard flow (i.e., generalized geodesic flow that reflects on the boundary) is ergodic, then the corresponding quantum ergodicity was proved by G\'erard-Leichtnam \cite{GL} and Zelditch-Zworski \cite{ZZ}. 

Comparing with the boundary-less case \cite{Sn, Ze, CdV}, the Riemann
moduli spaces are incomplete and the Weil-Petersson geodesic flow is
\textit{not} defined everywhere. This difference is reflected in the
structure assumptions (S1)-(S3) that we make later. Comparing with the
manifolds with boundary \cite{GL, ZZ}, the required analysis for the
proof of quantum ergodicity, e.g., the Egorov theorem in Theorem
\ref{thm:egorov}, is not available in the literature. We believe this
formulation of Egorov's theorem may be of independent interest.  (See
also the analytic assumptions (A1)-(A5).)

In fact, we prove Theorem \ref{thm:QE-moduli} for a more general class of singular spaces satisfying a number of structural and analytic hypotheses; in Section~\ref{sec:riem-moduli-spac} we observe that the Riemann moduli spaces $\mathcal{M}_{\gamma,n}$ satisfies these hypotheses.

Let $\Phi_{t}$ denote the flow generated by the Hamilton vector field of the homogeneous degree $1$ function $(x,\xi) \mapsto|\xi|_{g(x)}$. This function is (for now, formally) the principal symbol of the operator $P=\sqrt{\lap}$.   

The asymptotic behavior of the Laplacian eigenfunctions is closely
related to the dynamical properties of $\Phi_t$. Notice that in our
setting of singular spaces, the flow $\Phi_t(x,\xi)$ is not generally
defined for all $(x,\xi)\in T^{*}\Mreg\setminus0$, the cotangent space
of $\Mreg$ (removing the zero section). To clarify the notion of
distance from the singular locus, it is convenient to assume $\Mreg$
has a compactification. (In the examples considered in this paper,
compactifications are readily available.)

In particular, we assume the following structural properties of $\Mreg$:
\begin{enumerate}
\item[(S1).] There is a compact metric measure space $\Mfull$ such
  that $\Mfull \supset \Mreg$ and the closure of $\Mreg$ is $\Mfull$.
  For $x \in M$ and neighborhoods $U$ of $x$ sufficiently small, the
  measure and distance function correspond with the Riemannian measure
  of $(M, g)$.
\item[(S2).] The ``singular locus'' $\mathcal{P} = \Mfull \setminus
  \Mreg$ is closed.  Moreover, $\mathcal{P}$ has measure zero.
\item[(S3).] The distance function on $\Mreg \times \Mreg$ extends to a metric on $\Mfull \times \Mfull$. That is, the following function $d$ on $\Mfull \times \Mfull$ is a metric:
$$d(x,y)=\inf\left\{ \int_{0}^{1}|\gamma'(t)|_{g(\gamma(t))}\,dt\right\},$$
in which the infimum is taken from all smooth curve
$\gamma:[0,1]\to\Mfull$ such that $\gamma(0)=x$, $\gamma(1)=y$, and
$\gamma^{-1}(\mathcal P)$ has measure zero.  
\end{enumerate}

In practice, many of the compactifications used are larger than
required by our hypotheses and the distance function is degenerate on
the boundary of $\Mfull$, but assumption (S3) is satisfied after
passing to the quotient by the equivalence relation defined by $d$.

We may therefore define, for $\epsilon>0$, the spaces cut away from the singular locus $\mathcal P$:
\begin{equation*}
  \Mreg_{\epsilon} = \{ x \in \Mreg : d(x, \mathcal{P}) > \epsilon\}.
\end{equation*}
Observe that for $(x,\xi)\in T^*\Mreg\setminus0$, $\Phi_t(x,\xi)$ is
defined a priori only for $t\in\R$ for which
$d(\pi(\Phi_t(x,\xi)),\mathcal P)>\epsilon$ with some
$\epsilon>0$. Here, $\pi:T^*\Mreg\to\Mreg$ is the projection map.
Note that our assumptions above imply that $\Mreg_\epsilon \Subset \Mreg$,
since it is obviously compact in $\Mfull$ and its closure (the points
of distance at least $\epsilon$ from $\mathcal{P}$) is contained in $\Mreg$.

Due to the homogeneity of the geodesic flow, we need only study its restriction on the cosphere bundle $S^*\Mreg=\{(x,\xi)\in T^*\Mreg:|\xi|_{g(x)}=1\}$. We define, for each $q = (x, \xi) \in S^{*}\Mreg$, the maximum lifespan $T_{q}$ of the flow, i.e.,
\begin{equation*}
  T_{q} = \sup \left\{ T' \in [0,\infty]:\pi(\Phi_{t}(q)) \in \Mreg \text{
    for all } |t| \leq T'\right\}.
\end{equation*}
As in Zelditch-Zworski~\cite[Equation~2.5]{ZZ}, we also define the permissible sets $X_{T}$ and exceptional set $\mathcal{Y}$:
\begin{equation}\label{eq:XTY}
  X_{T} = \{ q \in S^{*}\Mreg \ : \ T_{q} \geq T\}, \quad \mathcal{Y}
  = S^{*}\Mreg \setminus \left(\bigcap_{T\in (0,\infty)}X_{T} \right).
\end{equation}
The exceptional set $\mathcal{Y}$ can be thought of (in the cases
considered below, quite concretely) as the flowout of the singular
locus.  If $(x,\xi) \notin \mathcal{Y}$, then $\Phi_{t}(x,\xi)$ exists for all $t\in\R$.

Now we make the following analytic assumptions about the manifold
$(\Mreg,g)$, which are verified for the examples of moduli spaces and
manifolds with conic singularities in Sections
\ref{sec:riem-moduli-spac} and \ref{sec:hyperb-surf-with}.  
\begin{enumerate}
\item[(A1).] $\Vol(\Mreg)<\infty$, where $\Vol$ is the volume with respect to the metric $g$.
\item[(A2).] For the (positive) Laplacian $\lap = \lap_{g}$ a
  self-adjoint extension $(\lap_{g}, \mathcal{D})$ (which we fix and denote below also by
  $\lap_{g}$) with core domain the $C^\infty_0(M)$ is chosen so that
  $\lap_{g}$ has compact resolvent, i.e. there is an operator $G
  \colon L^2 \lra \mathcal{D}$ such that $\Delta G - Id$ is compact
  and $G$ is compact on $L^2(M)$.  (As a result, its spectrum is discrete and
  consists only of eigenvalues $\lambda_{j}^{2}\to\infty$ as
  $j\to\infty$.)
\item[(A3).] The eigenvalues of $\lap$ obey a Weyl law, i.e.,
  \begin{equation*}
    N(\Lambda) = \# \{\lambda_{j} : \lambda _{j} \leq \Lambda\} =
    \frac{\Vol(\Mreg)\Vol (B_{n})}{(2\pi)^{n}} \Lambda^{n} + \lo (\Lambda^{n}),
  \end{equation*}
in which $\Vol (B_{n})$ denotes the volume of the unit ball in $\reals^{n}$ with respect to the Euclidean metric.
\item[(A4).] The set $\mathcal{Y}$ has Liouville measure zero in
  $S^{*}\Mreg$.
\item[(A5).] The geodesic flow on $X_{\infty}=\mathcal{Y}^{c}= \Mreg \setminus
  \mathcal{Y}$ is ergodic.   
\end{enumerate}

We remark that Assumptions (A1), (A2), and (A3) are enough to ensure that the heat operator $e^{-t\Delta}$ can be built via the functional calculus; this is useful to show that $\sqrt{\lap}$ is a pseudodifferential operator in
the region of interest. See Section \ref{sec:preliminaries} for
details. We also point out that assuming the Weyl law is only for
notational convenience; it has already been verified for Riemann
moduli spaces and is straightforward to verify (with current
technology of heat kernels) on manifolds with conic singularities.
We instead could impose an assumption on the small time behavior of
the heat kernel; though this hypothesis implies the Weyl law, in
practice it is sometimes easier to verify the Weyl law directly.

We may thus state our main theorem:
\begin{theorem} \label{thm:main-thm}  
  Suppose $(\Mreg, g)$ satisfies the structural (S) and analytic (A)
  assumptions above. If $\{\phi_{j}\}$ is an orthonormal basis of
  eigenfunctions of $\lap$ on $\Mreg$, then there is a density one
  subsequence $\{\phi_{j_k}\}\subset\{\phi_{j}\}$ so that
  \begin{equation*}
    \langle A \phi_{j_k} , \phi_{j_k} \rangle \to
    \int_{S^{*}M}\sigma_{0}(A) \,d\mu\quad\text{as }k\to\infty
  \end{equation*}
  for all order zero pseudodifferential operators $A$ with Schwartz kernel compactly
  supported in $\Mreg \times \Mreg$.  
\end{theorem}

\section*{Acknowledgements}
DB was partially supported by NSF grant DMS-1500646 and NSF CAREER
grant DMS-1654056.  JGR was supported by the Australian Research
Council through grant DP180100589. All the authors wish to thank the Australian Mathematical Sciences
Institute and the Mathematical Sciences Institute at the Australian
National University for their partial funding of the workshop
``Microlocal Analysis and its Applications in Spectral Theory,
Dynamical Systems, Inverse Problems, and PDE'' at which part of this
project was completed.

\section{Preliminaries}
\label{sec:preliminaries}

Let $\Mreg$ be a manifold that satisfies the structural and analytic
assumptions defined in the introduction. In this section, we gather
the facts about the microlocal analysis on such manifolds that are
required to prove quantum ergodicity in Theorem \ref{thm:main-thm}.
Because the singular structure on $\Mfull$ (i.e., the presence of the
singular locus $\mathcal{P}$) may be quite complicated, working near
$\mathcal{P}$ in principle would require a specialized
pseudodifferential calculus for each example (e.g., the
$\mathrm{b}$-calculus in the case of conic singularities; see
Hillairet--Wunsch~\cite{HW}).  However, in
Theorems~\ref{thm:QE-moduli} and~\ref{thm:main-thm}, we restrict our
analysis to pseudodifferential operators supported away from
$\mathcal{P}$.  Analysis in this region requires knowing little about
the precise structure of the singularities.

We use the correspondence of the pseudodifferential operators $A\in\Psi^m(\Mreg)$ of order $m$ and their principal symbols $\sigma_m(A)\in S^m(\Mreg)/S^{m-1}(\Mreg)$. We assume that the symbols have classical expansion at fiber infinity and therefore can be identified by functions in $C^\infty(S^*\Mreg)$ (so called the ``classical symbols''). See, e.g., H\"ormander \cite[Section 18.1]{H3} for detailed background. 

As in Zelditch--Zworski~\cite[Lemma 4]{ZZ}, we have a local Weyl law:

\begin{lemma}[Local Weyl law]\label{lem:local-Weyl}
Let $K \Subset M$ be a smooth manifold with boundary compactly contained in $M$ so that $K \setminus \pd K$ is an open domain, and let $A \in \Psi^{0}(K)$ have compactly supported Schwartz kernel. Then
\begin{equation*}
    \frac{1}{N(\Lambda)} \sum_{\lambda_{j}\leq \Lambda}\langle
    A\phi_{j} , \phi_{j} \rangle \to \int_{S^{*}M}\sigma_0(a)\,d\mu\quad\text{as }\Lambda\to \infty.
\end{equation*}
\end{lemma}

\begin{proof}
  This is a standard proof based on the short time estimate of the
  wave kernel $\cos(t\sqrt\lap)$. See Sogge \cite[Theorem 5.2.3]{So}
  and also H\"ormander \cite[Theorems 29.3.2 and 29.3.3]{H4} (for a
  proof of the Weyl law). Since the Schwartz kernel of $A$ is compactly
  supported, finite speed of propagation implies that
  $A\cos(t\sqrt\lap)A^\star$ still has compactly supported Schwartz
  kernel (i.e., support away from the singular locus $\mathcal P$)
  when $|t|$ is small enough. Therefore, the result of
  Sogge~\cite[Theorem 5.2.3]{So} applies.
\end{proof}

As a corollary, we have the following spatial version of the local Weyl law.
\begin{cor}
  \label{lem:local-weyl-cor}
  For every $f \in C^{\infty}_c(\Mreg)$, we have
  \begin{equation*}
    \frac{1}{N(\Lambda)}\sum_{\lambda_{j}\leq \Lambda}
    \int_{\Mreg}f(x) |\phi_{j}(x)|^{2} \,dV \to \int_{\Mreg} f(x)\,dV\quad\text{as }\Lambda\to\infty,
  \end{equation*}
  where $dV$ is the volume measure associated to the metric $g$.
\end{cor}

\begin{rmk}
  On compact manifolds, the Weyl law readily follows by taking $f=1$
  in the above corollary, c.f. Sogge \cite[Section 5.3]{So}. However,
  in our case of manifolds with singular locus $\mathcal P$,
  $f \in C^{\infty}_c(\Mreg)$ has to stay away from $\mathcal
  P$. Hence, the local Weyl law in Lemma \ref{lem:local-weyl-cor} does
  \textit{not} immediately imply the Weyl law, explaining its presence
  as assumption (A3).
\end{rmk}

We next provide a supplement of Egorov's theorem in Theorem
\ref{thm:egorov}, which is sufficient for the proof of quantum
ergodicity. We first require the following lemma establishing an
analogue of the off-diagonal smoothing property of pseudodifferential
operators.  The statement and proof of the lemma are essentially from
Hillairet--Wunsch \cite[Appendix A]{HW}.  There the authors assume that
the Friedrichs extension for the Laplacian is chosen, and we include
the proof here to clarify to the reader that the lemma holds for other
extensions (under our analytic and structural assumptions.)

\begin{lemma}\label{lem:HW}
Recall that $\mathcal P$ is the singular locus and $\Mreg_{\epsilon} = \{ x \in \Mreg : d(x, \mathcal{P}) > \epsilon\}$, i.e., the regular part of $\Mreg$ with distance at least $\epsilon$ from $\mathcal P$. 
\begin{enumerate}
  \item Suppose $0 <\epsilon ' < \epsilon$ and set $U = \Mreg_{\epsilon}$.
    For $V\subset \Mreg$ open with $\overline{V}\cap \Mreg_{\epsilon'} =
    \emptyset$, $\lap ^{N}\sqrt{\lap}$ is a bounded
    operator $L^{2}(V)\to L^{2}(U)$ and $L^{2}(U)\to L^{2}(V)$ for any $N\in\mathbb{N}$.
  \item For $\chi \in C^{\infty}_{c}(\Mreg_{\epsilon})$,
    $\chi\sqrt{\lap}\chi \in \Psi^{1}(\Mreg)$.
  \end{enumerate}
\end{lemma}

\begin{proof}
  As in Hillairet--Wunsch~\cite[Appendix A]{HW}, both results follow
  from an understanding of the smoothing properties of the heat kernel
  and using the relationship\footnote{If $\lap$ has finitely many
    non-postive eigenvalues (as may be the case for extensions other than
    the Friedrichs one), then one should project off of the non-positive
    eigenspaces.  These projections satisfy the conclusions of the
    theorem and the rest of the argument carries through.}  between
  the heat kernel and $\sqrt{\lap}$:
  \begin{equation*}
    \sqrt{\lap} = \frac{\lap}{\Gamma(\frac{1}{2})}\int_{0}^{\infty}e^{-t\lap}t^{-\frac{1}{2}}\,dt.
  \end{equation*}

  Take $\rho \in C^{\infty}_{c}([0,\infty))$ so that $\rho \equiv 1$
  on $[0, 2t_{0}]$ for some $t_{0} > 0$ and write $\psi = 1-\rho$.
  The contribution near infinity is smoothing because
  \begin{equation*}
    \int_{0}^{\infty}e^{-t\lap}\psi(t) t^{-\frac{1}{2}}\,dt =
    e^{-t_{0}\lap}\int_{0}^{\infty}e^{-(t-t_{0})\lap}\psi(t) t^{-\frac{1}{2}}\,dt.
  \end{equation*}
  The boundedness of this term (and, indeed, its composition with any
  power of $\lap$) follows from the functional calculus.

  We must thus show the results with $\sqrt{\lap}$ replaced by
  \begin{equation*}
    \lap \int_{0}^{\infty} e^{-t\lap}\rho(t) t^{-\frac{1}{2}}\,dt.
  \end{equation*}
  As multiplication by $\lap$ does not change the first result (and
  changes the second statement in a straightforward way), it suffices
  to study
  \begin{equation}\label{eq:sqrtlapintegral}
    \int_{0}^{\infty} e^{-t\lap}\rho(t) t^{-\frac{1}{2}}\,dt.
  \end{equation}

  We now consider the first statement. Take $a \in L^{2}(U)$ and
  define the distribution $T_{a}\in \mathcal{D}'(\reals \times V)$ by
  \begin{equation*}
    \left( T_{a}, \phi(t) b(y)\right)_{\mathcal{D}'\times \mathcal{D}}
    = \int_{0}^{\infty} \langle a, e^{-t\lap}b \rangle_{L^{2}} \phi(t) \,dt.
  \end{equation*}
  
  Take $b\in L^{2}(V)$. Since the supports of $a$ and $b$ are disjoint, $\lim_{t\downarrow
    0}\langle a , e^{-t\lap}b \rangle = 0$ and therefore
  \begin{equation*}
    (\pd[t] + \lap_{y})T_{a} = 0 \quad \text{in }\mathcal{D}'(\reals
    \times V).
  \end{equation*}
  We may thus conclude that $T_{a}$ is smooth.

  As $T_{a}\equiv 0$ for $t < 0$, for any $a \in L^{2}(U)$ and $b\in
  L^{2}(V)$, the function
  \begin{equation*}
    t\mapsto \langle e^{-t\lap}a , b\rangle
  \end{equation*}
  is smooth on $[0, \infty)$ and vanishes to infinite order at $0$.
  In particular, for each $N$ and $k$, the quantity
  \begin{equation*}
    t^{-k}\langle \lap^{N}e^{-t\lap}a, b\rangle
  \end{equation*}
  is bounded on $(0,1]$.  By the principle of uniform boundedness, we
  therefore know
  \begin{equation*}
    \norm[L^{2}(U) \to L^{2}(V)]{\lap^{N}e^{-t\lap}} = O(t^{k})
  \end{equation*}
  as $t\downarrow 0$ with a similar statement holding as a map
  $L^{2}(V) \to L^{2}(U)$.  Substituting this bound into the integral
  above yields the first result.

  For the second result, we fix a smooth Riemannian manifold
  $(\tilde{M}, \tilde{g})$ so that $\Mreg_{\epsilon}$ embeds
  isometrically as an open subset of $\tilde{M}$.  Let $e$ be the heat
  kernel on $\Mreg$ and $\tilde{e}$ be the heat kernel on $\tilde{M}$.
  Let $r$ denote the distribution on $\reals \times \Mreg_{\epsilon}
  \times \Mreg_{\epsilon}$ defined by
  \begin{equation*}
    \left( r, \phi\right) =
      \int_{0}^{\infty}\int_{\Mreg_{\epsilon}}\int_{\Mreg_{\epsilon}} \left(
        e(t,x,y) - \tilde{e}(t,x,y)\right) \phi(t,x,y) \,dy\,dx\,dt.
  \end{equation*}
  For any $\phi \in C^{\infty}_{c}(\reals \times \Mreg_{\epsilon}\times
  \Mreg_{\epsilon})$, we have
  \begin{equation*}
    \lim_{t\downarrow 0} \int_{\Mreg_{\epsilon}}\int_{\Mreg_{\epsilon}} \left(
      e(t,x,y) - \tilde{e}(t,x,y) \right) \phi(t,x,y) \,dx\,dy  = 0,
  \end{equation*}
  so, in $\mathcal{D}'(\reals \times \Mreg_{\epsilon} \times
  \Mreg_{\epsilon})$, we have
  \begin{equation*}
    \left( 2 \pd[t] + \lap_{x} + \lap_{y}\right) r = 0
  \end{equation*}
  and therefore $r$ is smooth on $\reals \times \Mreg_{\epsilon} \times
  \Mreg_{\epsilon}$.  We may thus replace $e^{-t\lap}$ with the heat kernel $\tilde{e}$ in \eqref{eq:sqrtlapintegral} and incur only an error of the form
  \begin{equation*}
    \int_{0}^{\infty}\rho(t) r(t,x,y) t^{-\frac{1}{2}}\,dt.
  \end{equation*}
  As $r$ is smooth and vanishing to infinite order at $t=0$, this
  integral is smoothing.  It therefore follows that $\chi
  \sqrt{\lap}\chi \in \Psi^{1}(\Mreg_{\epsilon})\subset \Psi^{1}(\Mreg)$.  
\end{proof}

Because $\lap$ has compact resolvent by the analytic assumption (A2), we obtain the following corollary.
\begin{cor}
  \label{lem:compact-off-diagonal}
  Fix $\epsilon > 0$ and let $\chi_{1} \in
  C^{\infty}_{c}(\Mreg_{\epsilon})$ and $\chi_{2} \in
  C^{\infty}_{c}(\Mreg)$ be such that $\chi_{2} \equiv 1$ on
  $\Mreg_{\epsilon}$.  The compositions $(1-\chi_{2}) P \chi_{1}$ and
  $\chi_{1} P (1-\chi_{2})$ are compact operators $L^{2}(\Mreg) \to
  L^{2}(\Mreg)$.  
\end{cor}

We now discuss the crucial Egorov's theorem. In general, Egorov's theorem connects the quantum evolution $e^{-itP} A e^{itP}$ and the classical evolution $\sigma_m(A)\circ\Phi_t$, where $A\in\Psi^m$ and recall that $P=\sqrt\lap$. Indeed, $e^{-itP} A e^{itP}\in\Psi^m$ and $\sigma_m(e^{-itP} A e^{itP})=\sigma_m(A)\circ\Phi_t$ on compact manifolds, see e.g. Sogge \cite[Theorem 4.3.6]{So}. 

In our setting of the singular space $\Mfull$, assume that $A$ has compactly supported Schwartz kernel in $\Mreg\times\Mreg$. Observe that $e^{-itP} A e^{itP}$ may not have compactly supported Schwartz kernel (so can potentially be close to the singular locus). We provide the following supplement to Egorov's theorem to remedy this issue. It is also of independent interest in the context of singular spaces.

As is standard, we let $\WF(A)$ denote the microsupport of $A$ (or
equivalently, the essential support of its symbol) and $\kappa_A$ be
the Schwartz kernel of $A$. (See \cite[Section 18.1]{H3} for more
background.)  We also note that if $a\in C_c^\infty(S^*\Mreg)$, then
there is $\tilde A\in\Psi^0(\Mreg)$ such that $\sigma_0(\tilde A)=a$
and $\kappa_{\tilde A}$ has compact support in $\Mreg\times\Mreg$. In
fact, let $A\in\Psi^0(\Mreg)$ such that $\sigma_0(A)=a$. Take
$\tilde A=\chi A\chi$ such that $\chi=1$ on $\pi(\supp(a))$. Then
$\tilde A-A$ is a smoothing operator.

\begin{theorem}\label{thm:egorov}
Let $\epsilon>0$ and $T>0$. Suppose that $A \in \Psi^{0}(\Mreg)$ has $\supp \kappa_{A} \subset \Mreg_{\epsilon} \times\Mreg_{\epsilon}$ and $\WF (A) \subset X_{T+\epsilon}$ defined in \ref{eq:XTY}. Let $\tilde A(t)\in\Psi^0(\Mreg)$ have compactly supported Schwartz kernel and $\sigma_0(\tilde A(t))=a\circ\Phi_t$ for $|t|\le T+\epsilon$.

Then for all $|t| \leq T$, 
$$e^{itP} A e^{-itP}-\tilde A(t):L^{2}(\Mreg)\to L^{2}(\Mreg)$$ 
is compact. 
\end{theorem}

\begin{proof}
Let $\delta > 0$ be such that the Schwartz kernels of $A$ and $\tilde{A}(t)$ lie in $\Mreg_{\delta}\times\Mreg_\delta$ for all $|t| \leq T+\epsilon$. Fix $0 < \delta ' < \delta$ and take $\chi_{1} \in C^{\infty}_{c}(\Mreg_{\delta'})$ be so that $\chi_{1} \equiv 1$ on $\Mreg_{\delta}$.  We also take $\chi_{2} \in C^{\infty}_{c}(\Mreg)$ so that $\chi_{2}\equiv 1$ on $\Mreg_{\delta'}$.

Consider the difference
  \begin{equation*}
    E(t) = e^{-itP}\tilde{A}(t) e^{itP} - A.
  \end{equation*}
It is then obvious that $E(0):L^{2}(\Mreg)\to L^{2}(\Mreg)$ is smoothing. Because
  $(1-\chi_{2})\tilde{A}(t) = \tilde{A}(t) (1-\chi_{2}) = 0$, we write
  \begin{align*}
    E'(t) &= e^{-itP}\left( \tilde{A}'(t) - i \left[ P,
            \tilde{A}(t)\right]\right) e^{itP} \\
    &= e^{-itP} \chi_{2} \left( \tilde{A}'(t) - i \left[ P,
      \tilde{A}(t)\right]\right)\chi_{2} e^{itP} \\
    &\quad - ie^{-itP}(1-\chi_{2}) P \tilde{A}(t) \chi_{2}  e^{itP} \\
    & \quad  + i e^{-itP}\chi_{2} \tilde{A}(t)  P (1-\chi_{2})  e^{itP}.
  \end{align*}
  Because the principal symbol of the inner part of the first term
  vanishes, we can write it as
  $e^{-itP}\chi_{2}R_{1}(t)\chi_{2}e^{itP}$, where
  $R_{1}(t)\in \Psi^{-1} (\Mreg)$.

  As $\chi_{1}\chi_{2} = \chi_{1}$ and $\tilde{A}(t)$ is supported
  where $\chi_{1}(x) \chi_{1}(y) \equiv 1$, the last two terms can be
  written
  \begin{equation*}
    - i e^{-itP} (1-\chi_{2}) P \chi_{1} \tilde{A}(t) \chi_{2} e^{itP}
    + i e^{-itP} \chi_{2} \tilde{A}(t) \chi_{1} P (1-\chi_{2}) e^{itP}.
  \end{equation*}

  We may therefore write the difference of interest as
  \begin{align*}
    e^{itP} A e^{-itP} - \tilde{A}(t) &= \int_{0}^{s}
                                        \chi_{2}R_{1}(s)\chi_{2}\,ds
    \\
                                      &\quad - i\int_{0}^{t} (1-\chi_{2})P
                                        \chi_{1}\tilde{A}(s)\chi_{2}
                                        \,ds + i \int_{0}^{t}\chi_{2}\tilde{A}(s)\chi_{1}P(1-\chi_{2})\,ds.
  \end{align*}
  The first term lies in $\Psi^{-1}(\Mreg)$ and has compactly
  supported Schwartz kernel; it is therefore compact on $L^{2}$.  The
  second two terms are both compact by
  Corollary~\ref{lem:compact-off-diagonal}. 
\end{proof}

\begin{rmk}
From the proof above, we observe that the compact operator $e^{itP} A e^{-itP}-\tilde A(t)$ is uniformly controlled for all $|t|\le T+\epsilon$.
\end{rmk}

\section{Proof of the main theorem}
\label{sec:proof-main-theorem}

We now show that under our assumptions, a modified version of the argument of Zelditch--Zworski~\cite[Section~3]{ZZ} still holds. Recall that $P = \sqrt{\lap}$.

We first establish some notation: For $B \in \Psi^{0}(\Mreg)$ with compactly supported Schwartz kernel and $T>0$, set
\begin{equation*}
  \rho_{j}(B) = \langle B \phi_{j} , \phi_{j}\rangle \quad\text{and}\quad \langle
  B\rangle_{T} = \frac{1}{2T}\int_{-T}^{T} e^{-itP}  B 
  e^{itP} \,dt.
\end{equation*}
Note that by Lemma~\ref{thm:egorov}, if $B$ has compactly supported Schwartz kernel and $\WF(B)$ is microsupported in $X_{2T+\epsilon}$, then with $\tilde{B}(t)$ as in
Lemma~\ref{thm:egorov} and
\begin{equation}\label{eq:average}
  \widetilde{\langle B \rangle}_{T} = \frac{1}{2T}\int_{-T}^{T}\tilde{B}(t)\,dt,
\end{equation}
we have that $\langle B \rangle _{T} - \widetilde{\langle B\rangle}_{T}:L^{2}(\Mreg)\to L^{2}(\Mreg)$ is compact.

Let $A \in \Psi^{0} (\Mreg)$ and write $a = \sigma (A)$. Suppose that $a\in C_c^\infty(S^*\Mreg)$ and $\kappa_A$ is compactly supported in $\Mreg \times\Mreg$. Set
\begin{equation*}
  \alpha = \int _{S^{*}M}a\quad\text{and}\quad \langle
 a\rangle_{T} = \frac{1}{2T}\int_{-T}^{T} a\circ\Phi_t \,dt,
\end{equation*}
where we are careful to use the second notation only for $a$ supported
in $X_{T+\epsilon}$.
The theorem then follows from a standard extraction procedure (see e.g. Zelditch-Zworski \cite{ZZ}) if we can show that
\begin{equation}\label{eq:twomoment}
  \frac{1}{N(\Lambda)} \sum_{\lambda_{j} \leq \Lambda} \left| \langle
    A\phi_{j}, \phi_{j}\rangle - \alpha \right|^{2} \to 0,
\end{equation}
as $\Lambda\to \infty$.

In the case where $\alpha = 0$, the proof essentially proceeds by a
series of approximations (the general case is proved fully below):
\begin{enumerate}
\item We replace $A$ by a family $A_{\epsilon, T}$ that have microsupport in the set $X_{2T+\epsilon}$. The difference of \eqref{eq:twomoment} for $A$ and $A_{\epsilon, T}$ can be estimated using the
local Weyl law in Lemma \ref{lem:local-Weyl}.
\item We then replace $A_{\epsilon,T}$ by an averaged operator
  $\widetilde{\langle A_{\epsilon, T} \rangle}_{T}$ (as in in
  \eqref{eq:average}) with compactly supported Schwartz kernel. By
  Egorov's theorem in Theorem \ref{thm:egorov}, $\widetilde{\langle
    A_{\epsilon, T} \rangle}_{T}$ is (modulo a compact operator) a
  pseudodifferential operator with principal symbol $\langle \sigma_0(A_{\epsilon, T})\rangle_T$.
\item We finally use the dynamical condition of ergodicity in $\Mreg$ to show that $\langle \sigma_0(A_{\epsilon, T})\rangle_T\to0$ when $T\to\infty$.
\end{enumerate}

We now let $T>0$, which later is chosen large enough. Write $U_{\epsilon} = U_{\epsilon}(T)$ as
\begin{equation*}
  U_{\epsilon} = \{ (x,\xi) \in X_{2T+\epsilon} \ : d (\pi(\Phi_{t}(x,\xi))
  , \mathcal{P} ) > \epsilon \text{ for all }|t| < 2T+\epsilon \}.
\end{equation*}
Observe that if $\epsilon < \epsilon '$, then $U_{\epsilon'}\Subset U_{\epsilon}$. Moreover, $\bigcap_{\epsilon>0}U_{\epsilon}=X_{2T}$, which is defined in \eqref{eq:XTY}.

Because the $U_{\epsilon}$ have compact closure away from $\mathcal{P}$, we can find microlocal cutoffs to the $U_{\epsilon}$. Namely, take $E_{\epsilon}\in \Psi^0(\Mreg)$ with compactly supported Schwartz kernels such that $\sigma(E_\epsilon)=1$ on $U_\epsilon$. Then $\lim_{\epsilon\to0}\sigma_0(E_\epsilon)=1$ on $X_{2T}$. Let
\begin{equation*}
  A_{\epsilon} = E_{\epsilon} A, \quad \alpha_{\epsilon} =
  \int_{S^{*}M}\sigma_0 (A_\epsilon), \quad R_{\epsilon} = I - E_{\epsilon}.
\end{equation*}

We now compare \eqref{eq:twomoment} for $A$ and $A_{\epsilon}$.  Write
\begin{equation}\label{eq:AAepsilon}
  C(\epsilon, \Lambda) = \frac{1}{N(\Lambda)} \sum_{\lambda_{j}\leq
    \Lambda} \left| \rho_{j}(A) - \alpha\right|^{2} -
  \frac{1}{N(\Lambda)}\sum_{\lambda_{j}\leq \Lambda} \left|
    \rho_{j}(A_{\epsilon}) - \alpha_{\epsilon}\right|^{2}.
\end{equation}
Note that $A=A_\epsilon+R_\epsilon A$. Letting $\beta_\epsilon =
\int_{S^{*}M}\sigma_{0}(R_{\epsilon}A)$, we have by the
Cauchy--Schwarz inequality,
\begin{eqnarray*}
C(\epsilon,\Lambda)&\le&\frac{2}{N(\Lambda)}\left(\sum_{\lambda_{j}\leq \Lambda}\left| \rho_{j}(A_{\epsilon}) -\alpha_{\epsilon}\right|^{2}\right)^{1/2}\left(\sum_{\lambda_{j}\leq \Lambda}\left| \rho_{j}(R_{\epsilon}A)-\beta_\epsilon)\right|^{2}\right)^{1/2}\\
&&+\frac{1}{N(\Lambda)}\sum_{\lambda_{j}\leq\Lambda}\left| \rho_{j}(R_{\epsilon}A) - \beta_\epsilon\right|^{2}\\
&\le&2\left(\frac{1}{N(\Lambda)}\sum_{\lambda_{j}\leq \Lambda}\rho_{j}((A_{\epsilon} -\alpha_{\epsilon})^*(A_{\epsilon} -\alpha_{\epsilon}))\right)^{1/2}\\
&&\times\left(\frac{1}{N(\Lambda)}\sum_{\lambda_{j}\leq \Lambda}\rho_{j}((R_{\epsilon}A-\beta_\epsilon)^*(R_{\epsilon}A-\beta_\epsilon))\right)^{1/2}\\
&&+\frac{1}{N(\Lambda)}\sum_{\lambda_{j}\leq\Lambda} \rho_{j}((R_{\epsilon}A -\beta_\epsilon)^*(R_{\epsilon}A- \beta_\epsilon)).
\end{eqnarray*}
Because the products $R_{\epsilon}A $ have compactly supported Schwartz kernel (since $A$ does), the local Weyl law of Lemma~\ref{lem:local-Weyl} shows that
$$\frac{1}{N(\Lambda)}\sum_{\lambda_{j}\leq\Lambda} \rho_{j}((R_{\epsilon}A -\beta_\epsilon)^*(R_{\epsilon}A- \beta_\epsilon))\to\left|\sigma_0(R_\epsilon A)-\beta_\epsilon\right|^2\quad\text{as }\Lambda \to \infty.$$
Therefore, using the trivial bound that $\rho_{j}((A_{\epsilon} -\alpha_{\epsilon})^*(A_{\epsilon} -\alpha_{\epsilon}))\le1$, we have that
\begin{equation}\label{eq:CepsilonLambda}
  C(T,\epsilon,\Lambda) = h_{T}(\epsilon) + r_{T,\epsilon}(\Lambda),
\end{equation}
where $r_{T,\epsilon}(\Lambda) \to 0$ as $\Lambda \to \infty$.
Because $\alpha(R_{\epsilon}A) \to 0$ and $\beta_\epsilon\to0$ as $\epsilon \to 0$, we also
know $h_{T}(\epsilon) \to 0$ as $\epsilon \to 0$.

We now turn our attention to the estimation of \eqref{eq:twomoment} involving $A_{\epsilon}$ and
$\alpha_{\epsilon}$:
\begin{eqnarray}
 &&\frac{1}{N(\Lambda)} \sum_{\lambda_{j}\leq \Lambda}\left| \langle A_{\epsilon}\phi_{j}, \phi_{j}\rangle - \alpha_{\epsilon}\right|^{2}\nonumber\\
 &\leq&\frac{1}{N(\Lambda)}\sum_{\lambda_{j}\leq \Lambda}\rho_{j}\left( \langle A_{\epsilon} - \alpha_{\epsilon}\rangle_{T}^{*}\langle A_{\epsilon}-\alpha_{\epsilon}\rangle_{T}\right)\nonumber\\
&=&\frac{1}{N(\Lambda)}\sum_{\lambda_{j}\leq \Lambda}\rho_{j}(B_{\epsilon,T})
\label{eq:AepsilonB}.
\end{eqnarray}
Observe that because $A_{\epsilon}$ is microsupported in
$X_{2T+\epsilon}$, Lemma~\ref{thm:egorov} allows us to replace
$B_{\epsilon,T}$ with
\begin{equation*}
  \tilde{B}_{\epsilon, T} = \widetilde{\langle A_{\epsilon} -
    \alpha_{\epsilon}\rangle}_{T}^{*}\widetilde{\langle A_{\epsilon} - \alpha_{\epsilon}\rangle}_{T},
\end{equation*}
whose principal symbol is
$$\left|
    \frac{1}{2T}\int_{-T}^{T}(\sigma_{0}(A_{\epsilon})\circ \Phi_{t} -
    \alpha_{\epsilon})\,dt\right|^{2},$$
moreover, $B_{\epsilon,T}-\tilde B_{\epsilon,T}:L^2(\Mreg)\to L^2(\Mreg)$ is compact. It then follows that
\begin{equation}\label{eq:BtildeB}
  \frac{1}{N(\Lambda)}\sum_{\lambda_{j}\leq
    \Lambda}\rho_{j}(B_{\epsilon,T}) \leq
  \frac{1}{N(\Lambda)}\sum_{\lambda_{j}\leq \Lambda}
  \rho_{j}(\tilde{B}_{\epsilon,T}) + f_{\epsilon,T}(\Lambda),
\end{equation}
where $f_{\epsilon,T}(\Lambda) \to 0$ as $\Lambda\to \infty$.

Since $\tilde B_{\epsilon,T}$ has compactly supported Schwartz kernel, the local Weyl law in Lemma~\ref{lem:local-Weyl} implies that
\begin{equation*}
  \frac{1}{N(\Lambda)} \sum_{\lambda_{j}\leq
    \Lambda}\rho_{j}(\tilde{B}_{\epsilon,T}) - \int_{S^{*}M}\left|
    \frac{1}{2T}\int_{-T}^{T}(\sigma_{0}(A_{\epsilon})\circ \Phi_{t} -
    \alpha_{\epsilon})\,dt\right|^{2}\,d\mu = \lo (1)
\end{equation*}
as $\Lambda\to \infty$. Putting together with \eqref{eq:AAepsilon}, \eqref{eq:CepsilonLambda}, \eqref{eq:AepsilonB}, and \eqref{eq:BtildeB}, we arrive at
\begin{equation*}
  \frac{1}{N(\Lambda)}\sum_{\lambda_{j}\leq \Lambda}\left| \rho_{j}(A)
    - \alpha\right|^{2} \leq \int_{S^{*}M} \left|
    \frac{1}{2T}\int_{-T}^{T}(\sigma_{0}(A_{\epsilon})\circ \Phi_{t} -
    \alpha_{\epsilon})\,dt\right|^{2}\,d\mu +
  F_{\epsilon,T}(\Lambda) + h_{T}(\epsilon),
\end{equation*}
in which $F_{\epsilon,T}(\Lambda)=r_{\epsilon,T}(\Lambda)+f_{\epsilon,T}(\Lambda)\to0$ as $\Lambda\to\infty$ and $h_{T}(\epsilon)\to0$ and $\epsilon \to 0$. To control the first time on the right-hand-side, notice that 
\begin{equation*}
g_T(\epsilon)=  \left| \int_{S^{*}M}\left| \frac{1}{2T}\int_{-T}^{T}(a\circ \Phi_{t}
      - \alpha)\,dt\right|^{2}\,d\mu - \int_{S^{*}M} \left|
      \frac{1}{2T}\int_{-T}^{T}(\sigma_{0}(A_{\epsilon})\circ\Phi_t -
      \alpha_{\epsilon})\,dt\right|^{2}\,d\mu\right| \to0
\end{equation*}
as $\epsilon\to0$ by dominated convergence theorem, since $\sigma_0(A_\epsilon)\to a$ and $\alpha_\epsilon\to\alpha$ as $\epsilon\to0$. We then use the ergodicity of the geodesic flow to conclude 
$$e(T)=\int_{S^{*}M} \left|
      \frac{1}{2T}\int_{-T}^{T}(\sigma_{0}(A) -
      \alpha)\,dt\right|^{2}\,d\mu\to0\quad\text{as }T\to\infty.$$
In total,
$$\frac{1}{N(\Lambda)}\sum_{\lambda_{j}\leq \Lambda}\left| \rho_{j}(A)
    - \alpha\right|^{2} \leq e(T)+g_T(\epsilon)+F_{\epsilon,T}(\Lambda) + h_{T}(\epsilon).$$
Taking $T$ large, $\epsilon$ small, and $\Lambda$ large successively, we complete the proof.

\section{Riemann moduli spaces with the Weil--Petersson metric}
\label{sec:riem-moduli-spac}

We now recall the definition and relevant properties of the Riemann
moduli spaces and their Weil-Petersson metrics; in particular, we show
that they satisfy assumptions (S) and (A) from the introduction, and
thus, from Theorem \ref{thm:main-thm}, we conclude that Theorem
\ref{thm:QE-moduli} holds.

As in the introduction, let $\mathcal{M}_{\gamma, n}$ denote the space of
equivalence classes of complex structures on a fixed, closed surface
$\Sigma$ of genus $\gamma$ with $n$ marked points $C = \{ p_1, \dots,
p_n \} \subset \Sigma$, where two complex structures on $\Sigma$ are equivalent if
one is the pullback of the other via a diffeomorphism $\Sigma$ which fixes
$C$.   The set $\mathcal{M}_{\gamma, n}$ admits a natural
compactification $\overline{\mathcal{M}}_{\gamma, n}$, the
Deligne--Mumford compactification, which includes, in addition to
complex structures on $\Sigma$, the nodal curves which can be obtained
by degenerations of complex structures $\Sigma$.  Then
$\overline{\mathcal{M}}_{\gamma, n}$ is a compact, complex orbifold of
complex dimension $3 \gamma - 3 + n$.  Within
$\overline{\mathcal{M}}_{\gamma, n}$ there is a finite family of
complex codimension $1$ ``normally crossing'' divisors, i.e.\ complex codimension $1$
sub-orbifolds, $D_1, \dots, D_\kappa$, such that $\bigcup_{i = 1}^\kappa D_i =
\overline{\mathcal{M}}_{\gamma, n} \setminus \mathcal{M}_{\gamma, n}$,
and any finite intersection $ \cap_{i \in J} D_i$ with $J \subset \{
1, \dots, \kappa  \}$, there is a
neighborhood $U$ of this intersection and a finite-to-one
ramified holomorphic resolution $V \lra U$ with $V$ an open complex
manifold such the inverse image of $ \cap_{i \in J} D_i$ is defined by the
vanishing of $|J|$ non-degenerate holomorphic functions $z_i$ with
linearly independent differentials on the intersection.   For further
background on the definition of $\mathcal{M}_{\gamma, n}$ and its Deligne--Mumford
compactification see for example the expository paper of Vakil
\cite{Va}.

Let $\Mreg = \mathcal{M}_{\gamma, n, \reg}$ be the top dimensional stratum of
$\mathcal{M}_{\gamma, n}$, i.e.\ the set $\mathcal{M}_{\gamma, n}$
minus the orbifold points.  This is a dense open set in
$\overline{\mathcal{M}}_{\gamma, n}$.  Recall our assumption $3\gamma
+ n \geq 4 $, which in the case $n = 0$ assures that $\gamma \ge 2$.
The Weil-Petersson metric $g_{\WP}$,
typically defined initially on the Teichm\"uller space and descending to a
smooth metric on $\Mreg$, is the Riemannian metric given locally by
identification of the cotangent bundle of $\Mreg$ at a point in $M$
(i.e.\ an equivalence class of Riemann surfaces  $[(\Sigma, c)]$), with the space
of transverse-traceless holomorphic quadratic differentials on the
uniformizing complete, hyperbolic metric $g$ on $(\Sigma \setminus C, c)$
with cusp-type singularities at $C$; the inner product on this
cotangent space is then given by the $L^2$-pairing defined by $g$.
This metric has a well-known decomposition near the divisors; at the
intersection $\cap_{i \in J} D_i$, for
appropriately chosen (holomorphic) defining functions $z_i = |z_i|
e^{\sqrt{-1} \theta_i}$ as in the
previous paragraph and setting $s_i^2 = 1 / \log(1/|z_i|)$, we have
\begin{equation}
  \label{eq:1}
  g_{\WP}  = \sum_{i \in J} c ds_i^2 + c' s_i^6 d\theta_i^2 + h_\cap + O(s^2)
\end{equation}
where $c, c' > 0$ are constants, $h_\cap$ is an (orbifold) metric on
$\cap_{i \in J} D_i$ and $s^2 = \sum_{i\in J} s_i^2$.  This expansion was
originally suggested by the work of Masur \cite{Mas} and established
in the work of a number of authors, including
by Liu--Sun--Yau~\cite{LSY} Wolpert~\cite{W1, W2, W3, W4} and
Yamada~\cite{Y}.  The full polyhomogeneous regularity of the
Weil-Petersson metric at the divisors is proven in Mazzeo--Swoboda~\cite{MS}
and Melrose--Zhu~\cite{MZ}.

We can now begin to address the structural and analytic assumptions.
Indeed, for (S1) and (S2), $\Mfull =
\overline{\mathcal{M}}_{\gamma, n}$, so $\Mfull - \Mreg$ is a closed
measure zero subset of $\Mfull$, and (S3) and (A1) follow from the local form
of the metric.  Skipping ahead to (A4) and (A5), consider the geodesic
flow of for the Weil-Petersson metric, which is
defined locally on $\Mreg$.  A result of Wolpert \cite{W2} implies
(see \cite{BMW}) that the set $X_\infty \subset S^* M$ of points in the
cosphere bundle on which the geodesic flow is defined for all times is
full measure, so its complement $\mathcal{Y}$ is measure zero, i.e.\
(A4) holds, and as mentioned in the introduction, that (A5) holds is the well-known result of
Burns--Masur--Wilkinson \cite{BMW}.

It remains to discuss (A2) and (A3).  Recall that, as is shown in
\cite{Lo94,Pi95},  $\overline{\mathcal{M}}_{\gamma, n}$
is in fact a ``good'' orbifold, meaning there is a complex manifold
$\Mfull'$ and a finite group $S$ acting on $\Mfull'$ by biholomorphic
maps (possibly with fixed points) such that the quotient if
$\overline{\mathcal{M}}_{\gamma, n} = \Mfull'/S$ and the projection 
\begin{equation}
\pi \colon \Mfull' \lra \overline{\mathcal{M}}_{\gamma,
  n}\label{eq:mod-res}
\end{equation}
is a smooth (ramified)
holomorphic map.  The pullback of the Weil-Petersson metric $\pi^*
g_{\WP}$ to $\Mfull'$ is a smooth Riemannian metric on
$\Mreg' :=\pi^{-1}(\mathcal{M}_{\gamma, n})$, and elements of $S$ are
automatically isometries of this pullback metric.  For
$\gamma$ fixed and $n$ large, one can take $\Mfull' = \overline{\mathcal{M}}_{\gamma,
  n}$ as there are no fixed points of the action of the mapping class
group on Teichm\"uller space, see \cite{Va,
  JMMV}.

Ji--Mazzeo--M{\"u}ller--Vasy~\cite{JMMV} study the general class
of complex orbifolds $\Mfull$ which have ``crossing cusp-edge''
singularities in the metric.  These are exactly those complex
Riemannian orbifolds whose metrics take the form described in the
above paragraphs near a fixed set of normally intersecting complex
codimension one divisors.  In particular, they prove that Laplacian on
$\mathcal{M}_{\gamma, n}$ is self-adjoint with core domain $C^\infty_{0,
  \orb}(\mathcal{M}_{\gamma, n})$, the Frechet space of smooth
functions $\phi$ such that, with $\pi$ the resolving map from the
previous paragraph, $\phi \circ \pi \in C^\infty_0(\Mreg')$.  In
words, these are the functions which are compactly supported in
$\mathcal{M}_{\gamma, n}$, smooth away from all orbifold
singularities, and lift via the local resolutions of the orbifold
singularities to smooth functions.  They prove (see Theorem 3) that
with this core domain, $\lap_{g_{\WP}}$ is essentially self-adjoint,
that the domain of this self-adjoint extension is compactly contained
in $L^2$ (see below Theorem 3), and that Weyl asymptotics hold for the
(necessarily discrete) spectrum (see Theorem 1).  (We remark again
that in \cite{JMMV} all the statements are for the non-pointed moduli spaces
$\mathcal{M}_{\gamma}$ but all of the theorems in the body of the
paper are for the general class of singular Riemannian space which
include $\mathcal{M}_{\gamma, n}$.) In particular,
assumptions (A2) and (A3) hold for this extension.

Thus the assumptions (S) and (A) hold for $\lap_{g_{\WP}}$ on $\mathcal{M}_{\gamma, n}$
with its unique self-adjoint extension with core domain $C^\infty_{0,
  \orb}$, i.e.\ Theorem \ref{thm:QE-moduli} follows from Theorem
\ref{thm:main-thm}.

\subsection{Orbifold regular PsiDO's on $\mathcal{M}_{\gamma, n}$}  We
now prove a stronger theorem for the Riemann moduli space.  We
continue with the notation of the previous section, in particular $M =
\mathcal{M}_{\gamma, n, \reg}$, consider
pseudodifferential operators $A \in \Psi^0_{0,\orb}(\mathcal{M}_{\gamma,
  n})$ which, by definition, are operators $A \colon C^\infty_0(M)
\lra \mathcal{D}'(M)$ which have compactly supported Schwartz kernel in
$M$ and are regular under local orbifold
resolutions; concretely, for the resolving map $\pi$ in
\eqref{eq:mod-res}, $\pi^* A \in \Psi^0(M')$ (and $\pi^*A$ is
compactly supported in $M'$.)  Equivalently, working on the resolved
space $\Mreg'$, these are
pseudodifferential operators $A \in \Psi^0(\Mreg')$ with compactly
supported Schwartz kernels which are invariant under the action of $S$
on $\Mreg'$.  This family of pseudodifferential operators is defined
independently of a choice of resolution $\Mreg'$ as it is a equivalent
to smoothness of the pullback of the $A$ via any local resolution, but
below we use a particular convenient choice of resolution,
specifically the one used in \cite[Sec.\ 6]{BMW},
$$
\Mreg' = \mathcal{T} / \mathrm{MCG}[k],
$$
where $\mathcal{T}$ is the Teichm\"uller space and
$$
\mathrm{MCG}[k] = \{ \psi \in \mathrm{MCG}(\Sigma) : \psi_* \equiv 0
\mbox{ acting on } H^1(\Sigma; \mathbb{Z}/k\mathbb{Z}) \},
$$
is a finite
index subgroup of the mapping class group $\mathrm{MCG}(\Sigma)$ which
is obviously normal.  In \cite[Thm.\ 6.4]{BMW},
the authors prove that the Weil-Petersson geodesic flow is ergodic on
this resolved space, so since the flow is defined for infinite times
on the pullback of a full measure set, both assumptions (A4) and (A5)
hold on $\Mreg'$ with the Weil-Petersson metric.  The moduli space is
then the quotient of $\Mreg'$ by the set of biholomorphic maps
parametrized by (and identified with representatives of the set of)
the group $S = \mathrm{MCG}(\Sigma) / \mathrm{MCG}[k]$.  This will be useful
to prove the following.
\begin{theorem}\label{thm:stronger}
Assumptions as in Theorem \ref{thm:QE-moduli}, there is a density one subsequence
$\{\phi_{j_k}\}\subset\{\phi_{j}\}$ such that for all $A \in
\Psi^0_{0, \orb}(\mathcal{M}_{\gamma, n})$,
\begin{equation*}
\langle A\phi_{j_k}, \phi_{j_k} \rangle \to
\int_{S^{*}\Mreg}\sigma_{0}(A)\,d\mu\quad\text{as }k\to\infty.
\end{equation*} 
\end{theorem}
\begin{proof}
 Since $(\Mreg', \pi^* g_{\WP})$ is a smooth crossing cusp-edge space, the results of \cite{JMMV}
  show that $\lap_{\pi^* g_{\WP}}$ is essentially self-adjoint
  with core domain $C^\infty_0(\Mreg')$, and that assumptions
  (A2)--(A3) hold for this self-adjoint extension.   The rest of the
  assumptions (S) and (A) also follow.  Indeed, the assumptions (S)
  assumptions and  (A1) hold automatically, and assumptions (A4) and
  (A5) follow as discussed prior to the statement of the theorem.
  Thus all the hypotheses are satisfied and the conclusion of Theorem
  \ref{thm:main-thm} applies to $\lap_{\pi^* g_{\WP}} $.
  
  The theorem now follows easily from considering the identification
  of the eigenspaces $E_\lambda$ of $\lap_{g_{\WP}}$ on $\mathcal{M}_{\gamma, n}$
  for the unique self-adjoint extension from $C^\infty_{0, \orb}$ with
  the $S$-invariant eigenspaces of $\lap_{\pi^* g_{\WP}}$.  Indeed,
  let $\tilde E_\lambda$ denote an eigenspace of $\lap_{\pi^*
    g_{\WP}}$, and note that since $S$ acts on $(\Mreg', \pi^*
  g_{\WP})$ by isometries, it acts by pullback on $\tilde E_\lambda$.
  Letting $\phi \in \tilde E_\lambda$, then $\phi^S =
  |S|^{-1} \sum_{\psi \in S} \psi^* \phi$ is an $S$-invariant function
  on $\Mreg'$ and thus descends to a function on $\mathcal{M}_{\gamma,
    n}$ which it is easy to see lies in the domain under
  consideration.  The other direction of identification is automatic.
  Thus for all $\lambda \in \mathrm{spec}(\lap_{\pi^* g_{\WP}})$,
  $$
  \tilde E_\lambda^S := \{ \phi \in \tilde E_\lambda : \psi^* \phi =
  \phi \} \subset \tilde E_\lambda
  $$
  satisfies $\tilde E_\lambda^S = E_\lambda$ gives an identification
  of $E_\lambda$ with a subset of $\tilde E_\lambda$.  In particular,
  we may choose an orthonormal basis of eigenfunctions $\tilde \phi_j$
  of $(\Mreg', \pi^*  g_{\WP})$ which contains a subsequence of an
  orthonormal basis of $\tilde E_\lambda^S$.

  On the other hand, the Weyl asymptotic formulas implies that, if
  $\tilde N(\lambda)$ is the  eigenvalue counting function for $(\Mreg', \pi^*
  g_{\WP})$ and $N(\lambda)$ the counting function for
  $(\mathcal{M}_{\gamma, n}, g_\WP)$, 
  \begin{equation}
    \label{eq:2}
    \frac{\tilde N(\lambda)}{N(\lambda)} = \frac{\Vol (\Mreg', \pi^*
  g_{\WP})}{\Vol(\mathcal{M}_{\gamma, n}, g_\WP)} + o(1)  = |S| + o(1)\mbox{ as }
\lambda \to \infty.
\end{equation}
Hence any full density subsequence of eigenfunctions of $(\Mreg', \pi^*
  g_{\WP})$ contains a full density subsequence of eigenfunctions
  coming from the $\tilde E_\lambda^S$. Now there is a full density subsequence of eigenfunctions
  $(\tilde \phi_{j_k})$ which satisfy the conclusion of Theorem
  \ref{thm:QE-moduli}. It contains a subsequence of invariant
  eigenfunctions $(\tilde \phi_\ell^S) = (\tilde
  \phi_{j_k}) \cap L^2_S(\Mfull')$ that also satisfy the conclusion
  fo Theorem \ref{thm:QE-moduli} and in addition each $\tilde
  \phi_\ell^S = \pi^* \phi_\ell$ for some eigenfunction on
  $\mathcal{M}_{\gamma, n}$. Thus for any $B \in
  \Psi^0(\Mreg')$ with compact support, we have
\begin{equation*}
\langle B \tilde \phi_{\ell}^S, \tilde \phi_{\ell}^S \rangle \to
\int_{S^{*}\Mreg'}\sigma_{0}(B)\,d\mu\quad\text{as }\ell\to\infty.
\end{equation*} 
Taking $B = \pi^*A$ and dividing by the area gives the result. \end{proof}

\section{Hyperbolic surfaces with conic singularities}
\label{sec:hyperb-surf-with}

We consider the example of hyperbolic surfaces with conic
singularities.\footnote{We consider only surfaces for the sake of
  brevity; the same method likely extends to hyperbolic cone manifolds
of arbitrary dimension as described by McMullen~\cite{McMullen}.}  Concretely, consider a compact Riemann surface
$\overline{M}$ of genus $\gamma$, a finite set of points
$\mathcal{P}$.  Suppose $\overline{M}$ is equipped with a Riemannian
metric $g$ smooth on the complement
$\Mreg = \overline{M}\setminus \mathcal{P}$ and so that
\begin{enumerate}
\item for each $p \in \mathcal{P}$ there are conformal coordinates
  $\tilde{z}$ with $\tilde{z}(p) = 0$, 
\item in the (non-smooth) coordinates $z =
  \alpha^{-1}\tilde{z}^{\alpha}$, we have
  \begin{equation*}
    g = dr^{2} + \alpha^{2}\sinh^{2}r \, d\theta^{2}, 
  \end{equation*}
  where $z = re^{i\theta}$, and 
\item $g$ is hyperbolic on $\overline{M}\setminus \mathcal{P}$.
\end{enumerate}
Here $\alpha = 1$ corresponds to a ``phantom singularity''; in other
words, when $\alpha = 1$, the metric extends to be smooth at the point
$p$.

Given a finite set of points $\mathcal{P} = \{ p_{1}, \dots, p_{k}\}$
and numbers $\alpha_{1}, \dots , \alpha_{k} \in (0, \infty)$,
McOwen~\cite{Mc} showed the existence (and uniqueness) of a
hyperbolic metric on $\Mreg$ with conic singularities of the form
above at the points $p_{j}$ with constants $\alpha_{j}$.

The spectral theory and heat kernel asymptotics of various
self-adjoint extensions of the Laplacian $\lap_{g}$ (and the Laplace
operator on more general Riemannian spaces with conic singularities)
were studied originally by Cheeger~\cite{C}, with later works
including Lesch~\cite{L}, Mooers~\cite{Mo}, and
Gil--Mendoza~\cite{GM}.  In particular, the first three analytic
assumptions are well-known; see, for example, the book of
Lesch~\cite[Page 72]{L}.

We verify assumption (A4) directly; assumption (A5) follows from the
hyperbolicity of the metric (one can treat $\Mreg =
\overline{M}\setminus \mathcal{P}$ as an open hyperbolic system). See e.g. Brin \cite[Appendix]{BB} for a short and nice proof for ergodicity of Anosov geodesic flows.
\begin{lemma}
  \label{lem:a4-for-hyp}
  The set
  \begin{equation*}
    \mathcal{Y} = \{ (x,\xi) \in S^{*}\Mreg : \pi (\Phi_{t}(x,\xi))\in
    \mathcal{P} \text{ for some }t\in \reals\}
  \end{equation*}
  has measure zero.
\end{lemma}

\begin{proof}
  For $T > 0$, let
  \begin{equation*}
    Y_{\pm, T} = \{ (x,\xi) \in S^{*}\Mreg : \pi (\Phi_{t}(x,\xi) )
    \in \mathcal{P} \text{ for some } t, \pm t \in (0,T)\}.
  \end{equation*}
  For $T$ sufficiently small, $Y_{\pm, T}$ has measure zero by the
  model form of the metric.  We now realize $\mathcal{Y}$ as the
  countable union of flowouts of $Y_{\pm, T}$ and so it has measure zero.
\end{proof}

As $(\Mreg, g)$ satisfies the structural and analytic hypotheses, we
have the following corollary:
\begin{cor}
  \label{lem:hyperbolic-surfaces}
  If $(\Mreg, g)$ is a hyperbolic surface with conic singularities,
  then it is quantum ergodic as in Theorem~\ref{thm:main-thm}.
\end{cor}

\end{document}